\documentclass[11pt]{amsart} \textwidth=14.5cm \oddsidemargin=1cm
\usepackage{setspace}
\usepackage{amsthm}
\usepackage[rgb]{xcolor}
\usepackage{times}
\usepackage[linktocpage=true]{hyperref}
\usepackage{mathrsfs}

\usepackage{xypic}
\usepackage{verbatim}
\usepackage{amscd}
\usepackage{amsmath}
\usepackage{amsfonts}
\usepackage{amsxtra}
\usepackage{amssymb}
\usepackage{eucal}
\usepackage{latexsym,todonotes}
\usepackage{stmaryrd}
\usepackage{graphicx}%
\usepackage{dynkin-diagrams}
\usepackage{tikz}
\usetikzlibrary{intersections}
\usetikzlibrary{knots}
\newtheorem{theorem}{Theorem}[section]

\newtheorem{corollary}[theorem]{Corollary}

\newtheorem{definition}[theorem]{Definition}
\newtheorem{example}[theorem]{Example}

\newtheorem{lemma}[theorem]{Lemma}

\newtheorem{proposition}[theorem]{Proposition}
\newtheorem{remark}[theorem]{Remark}

\numberwithin{equation}{section}

\newcommand{\OS}{\mathcal{OS}(z,t)}
\newcommand{\OSq}{\mathcal{OS}(q-q^{-1},t)}
\newcommand{\OSqq}{\mathcal{OS}(q-q^{-1},t)}

\newcommand{\N}{\mathbb{N}}
\newcommand{\Z}{\mathbb{Z}}
\newcommand{\R}{\mathbb{R}}

\newcommand{\Q}{\mathbb{Q}}
\newcommand{\C}{\mathbb{C}}
\newcommand{\Hy}{\mathscr{H}}

\newcommand{\gl}{\mathfrak g \mathfrak l}

\newcommand{\la}{\langle}
\newcommand{\ra}{\rangle}

\newcommand{\U}{\mathbf U}

\newcommand{\I}{\mathbb I}

\newcommand{\K}{\mathbb K}

\newcommand{\ua}{\uparrow}
\newcommand{\da}{\downarrow}
\newcommand{\uda}{\la \ua,\da\ra}
\newcommand{\opp}[1]{\overset{\leftarrow}{#1}}
\newcommand{\Di}{\mathcal{P}_{{\bf a},{\bf b}}}
\newcommand{\nl}{m}
\newcommand{\nr}{n}

\newcommand{\WB}[3]{ B_{#1|#2}(#3)}
\newcommand{\qWB}[4]{ \mathfrak B_{#1|#2}(#3,#4)}
\newcommand{\HA}[1]{\Hy_{#1}}

\newcommand{\unit}[1]{1_{\bf #1}}

\allowdisplaybreaks[4]

\title{Canonical bases of the oriented skein category}
\author{Yaolong Shen}
\address{Department of Mathematics, University of Virginia, Charlottesville, VA 22904}
\email{ys8pfr@virginia.edu}

\begin{document}
\maketitle

\begin{abstract}
		We develop a bar involution and canonical basis for every morphism space of the oriented skein category through a diagrammatic approach. In particular, our construction gives rise to Kazhdan-Lusztig type bases on quantized walled Brauer algebras.
	\end{abstract}
	
	\maketitle

\section{Introduction}
The {\em walled Brauer algebra} $\WB{\nl}{\nr}{\delta}$ is a distinguished subalgebra of the classical Brauer algebra $B_{\nl+\nr}(\delta)$, where $\delta$ is a parameter. It was introduced independently by Koike \cite{Ko89} and Turaev \cite{Tur90}, motivated in part by a Schur duality between $\WB{\nl}{\nr}{\delta}$ and the general linear group $GL_k$ arising from mutually commuting actions on the mixed tensor space $V^{\otimes \nl}\otimes {V^*}^{\otimes \nr}$, where $V$ is the natural representation of $GL_k$; see also \cite{BCHLLS}. 

 Leduc \cite{Le94} introduced a two-parameter deformation, $\qWB{\nl}{\nr}{q}{t}$, of the walled Brauer algebra using (non-diagrammatic) generators and relations. In this article, we will refer to this deformation as the {\em quantized walled Brauer algebra}. Subsequently, Dipper, Doty, and Stoll \cite{DDS14} employed oriented tangles to generate $\qWB{\nl}{\nr}{q}{t}$ in terms of diagrams, resulting in the construction of an integral basis. In \cite{En03}, Enyang demonstrated that $\qWB{\nl}{\nr}{q}{t}$ exhibits a cellular structure. Building upon this cellular structure, Rui and Song \cite{RS15} provided a criterion for the semisimplicity of quantized walled Brauer algebras $\qWB{\nl}{\nr}{q}{t}$ and classified their simple modules over an arbitrary field.


The quantized walled Brauer algebra also appears in the study of the oriented skein category \cite{B17}. The category was initially introduced as the Hecke category in \cite[\S 5.2]{Tur90}, with a distinct normalization of crossings, while in \cite[Definition 2.1]{QS19}, it is referred to as the quantized oriented Brauer category. The oriented skein category is a ribbon category that underpins the definition of the HOMFLY-PT invariant of an oriented link, in the same way that the Temperley-Lieb category underpins the Jones polynomial.


In this article, we adopt a similar perspective by considering the quantized walled Brauer algebra as a morphism space within the oriented skein category $\OS$ and study the morphism spaces of $\OS$ in general.

The oriented skein category is a quotient of the category of framed oriented tangles. Its objects are represented as words using the ${\ua, \da}$ alphabet, and its morphisms are constructed using ribbons (see \S\ref{sec:def}). These ribbons consist of bubbles, strands with orientations, and crossings with signs. The basis theorem for the morphism spaces of $\OS$ has been well-established in \cite{Tur90}. For each morphism space of $\OS$, we select a specific basis formed by the set of "reduced" ribbons containing only positive crossings, referred to as the {\em standard basis}, see Proposition~\ref{prop:basis}.

The first observation we made in this article is that morphism spaces of $\OSqq$ also admit a unique bar involution $\psi$, which flips positive and negative crossings while leaving the cups and caps unchanged as detailed in Proposition~\ref{prop:bar}. We note that the type A Hecke algebra also arises as a special case of the morphism space of $\OSqq$ and  this bar involution coincides with the usual bar map of the type A Hecke algebra.

We define the length of a reduced ribbon to be the number of crossings within it. The second observation in this article is that when we apply $\psi$ to an element of the standard basis, the resulting ribbon can be expressed in terms of positive reduced ribbons with shorter lengths, and the leading coefficient is $1$. Therefore, it follows from the Lusztig's lemma (cf. \cite[Lemma 24.2.1]{Lu94}) that
every morphism space of $\OSqq$ also admits a Kazhdan-Lusztig type basis (called canonical basis), see Theorem~\ref{thm:cbhom}. 

An immediate consequence of the compatibility of the bar involutions is that the usual type $A$ Kazhdan-Lusztig basis is a part of the canonical basis that we obtain. Moreover, we observe that the transition matrices between the canonical basis and standard basis is independent of the variable $t$. The same phenomenon also appeared in \cite[\S 5.2]{FG95},\cite{CS22}.

It will be interesting to know whether the canonical basis (for specialized parameters) admits positivity. Evidence supporting such a positivity property has been found in some low-rank cases, as demonstrated in Example~\ref{ex:n2m1} and Example~\ref{ex:n3m1}. Additionally, it would be of significant interest to explore the connection between our canonical bases with the ones of quantum groups through the mixed Schur-Weyl duality.

The paper is organized as follows. In Section~\ref{sec:OS}, we provide an in-depth review of the oriented skein category and explain how to construct the bar involution and the canonical basis on its morphism spaces. In Section~\ref{sec:qWB}, we review the algebraic definition of the quantized walled Brauer algebra $\qWB{\nl}{\nr}{q}{t}$ and view it as a particular morphism space of $\mathcal{OS}(q-q^{-1},t)$. We then construct the bar involution pure algebraically and we give examples of the canonical bases in low rank cases. 


{\bf Acknowledgement.} The author thanks his advisor Weiqiang Wang for suggesting the topic and many constructive discussions. Additionally, he thanks Jon Brundan for his insightful suggestion on putting quantized walled Brauer algebras in the framework of the oriented skein category. Moreover, he thanks an anonymous expert for suggesting constructing the canonical basis for all morphism spaces of the oriented skein category under a diagrammatical viewpoint. The author is supported by Graduate Research Assistantship from Wang's NSF grant (DMS-2001351) and a semester fellowship from University of Virginia.

\section{Oriented skein category}
\label{sec:OS}
In this section, we recall the definition and basic properties of the oriented skein category. We then construct bar-involutions and Kazhdan-Lusztig type bases for all morphism spaces of this category.
\subsection{Definition}
\label{sec:def}
We first to recall briefly the definition of the category $\mathcal{FOT}$ of framed oriented
tangles, which is the framed analog of the oriented tangle category $\mathcal{OT}$ introduced by
Turaev in \cite{Tur90}. The objects of $\mathcal{FOT}$ are given by the set $\uda$ of all words in the alphabet $\{\ua,\da\}$. Tensor products of objects are given by concatenation and the unit object is the empty word $\varnothing$. For example we have $\ua \otimes \da \otimes \ua=\ua\da\ua$. 

For any two given words ${\bf a}=a_m\cdots a_1$ and ${\bf b}=b_n\cdots b_1\in \uda$, the morphism space $Hom_{\mathcal{FOT}}(\bf a,\bf b)$ is the collection of isotopy classes of framed
oriented tangles in $[0,1]\times [0,1]\times \R$ with boundary 
\[
\left\{
(\frac{m+1-i}{m+1},0,0)\mid i=1,\ldots,m
\right\}
\cup
\left\{
(\frac{n+1-j}{n+1},0,0)\mid j=1,\ldots,n
\right\}
\]
such that the orientation near the boundary points $(\frac{m+1-i}{m+1},0,0)$ and $(\frac{n+1-j}{n+1},0,0)$ are determined by $a_i$ and $b_j$, respectively. We will draw these tangles by projecting onto the $xy$-plane such that there is no triple crossing and we keep track the information of 'over' and 'under' data at each crossing. The resulting diagram will be referred to as {\em $(\bf a,\bf b)$-ribbons}. The tensor product of morphisms is given by horizontal concatenation while the composition of morphisms is given by vertical stacking. For example we have $f\otimes g=fg$ and $f\circ g=\overset{f}{g}$.

\begin{example}
Below is an example of a $(\da\ua\ua\da\ua\ua,\ua\da\ua\ua)$-ribbon.
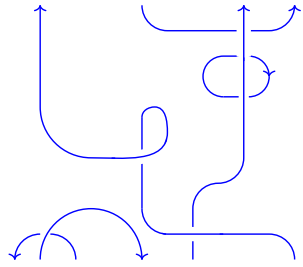
\[
    \resizebox{!}{100pt}{\begin{tikzpicture}[blue,
thick,decoration={
    markings,
    mark=at position 1 with {\arrow{>}}}
    ] 
    \begin{knot}[clip width=10,
clip radius=15pt,consider self intersections=true,]
\strand[blue,thick] (0,0) to (0,-2)
to [out=down,in=left] (1,-3)
to [out=right,in=down] (2.5,-2.5)
to [out=up,in=right] (2.25,-2)
to [out=left,in=up] (2,-2.2);
\strand[blue,thick] (2,-2.2) to (2,-4);
\strand[blue,thick] (2,-4) to 
[out=down,in=left] (2.5,-4.5)
to (4.5,-4.5)
to [out=right,in=up] (5,-5);
\strand[blue,thick] (3,-5)
to (3,-4)
to [out=up,in=left] (3.5,-3.5)
to [out=right,in=down] (4,-3)
to (4,0);
\strand[blue,thick] (2,0)
to [out=down,in=left] (2.5,-0.5)
to (4.5,-0.5)
to [out=right,in=down] (5,0);
\strand[blue,thick] (4.5,-1.4)
to [out=up,in=right] (4.1,-1)
to (3.6,-1)
to [out=left,in=up] (3.2,-1.4)
to [out=down,in=left] (3.6,-1.8);
\strand[blue,thick] (3.6,-1.8) to
(3.9,-1.8);
\strand[blue,thick]
(4.1,-1.8) to [out=right,in=down] (4.5,-1.4);
\strand[blue,thick] (0,-5) 
to [out=up,in=left] (1,-4)
to [out=right,in=up] (2,-5);
\strand [blue,thick] (-0.5,-5) to
[out=up,in=left] (0,-4.5);
\strand[blue,thick] (0.2,-4.5) to
[out=right,in=up] (0.7,-5);
\end{knot}
\draw[postaction={decorate}](0,-0.01)--(0,0);
\draw[postaction={decorate}](4,-0.01)--(4,0);
\draw[postaction={decorate}](5,-0.01)--(5,0);
\draw[postaction={decorate}](2,-4.99)--(2,-5);
\draw[postaction={decorate}](4.5,-1.399)--(4.5,-1.401);
\draw[postaction={decorate}](-0.5,-4.99)--(-0.5,-5);
\end{tikzpicture}}
\]
\end{example}

Moreover, the isotropy translates into the equivalence relations on diagrams generated by planar isotropy together with the oriented
Reidemeister Moves (FRI), (RII) and (RIII) as listed in \cite[Figure 1]{B17}. As justified carefully in \cite[\S I.4.6]{Tur16}, any $(\bf a,\bf b)$-ribbon is isotopic to a {\em generic} one, meaning that all of its critical points are local maxima and minima, and all crossings occur away from the critical points. From now on when we only consider generic ribbons. 

Now let $\K$ be any field and fix parameters $z,t\in \K^*$. Let $\mathscr I$ be the tensor ideal of the $\K$-linearization of $\mathcal{FOT}$ uniquely determined by the Conway skein relation (S), the twisting relation (T) and the free loop relation (L) as follows:
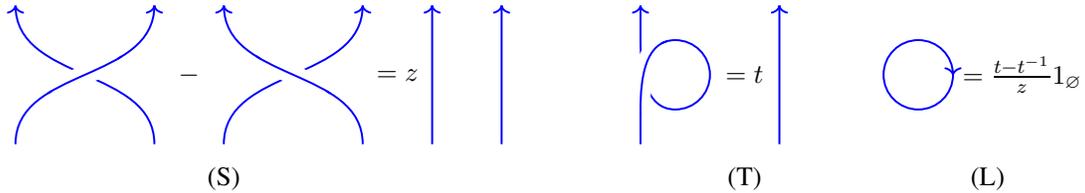
\begin{center}
\resizebox{\linewidth}{!}{
\begin{tikzpicture}[
thick,decoration={
    markings,
    mark=at position 1 with {\arrow{>}}}
    ] 
    \begin{knot}[clip width=10,
clip radius=15pt,
consider self intersections=true,]
\node at (2.5,1){$-$};
\node at (5.5,1){$=z$};
\strand[blue,thick] (0,0) 
to [out=up, in=down] (2,2);
\strand[blue,thick] (2,0) 
to [out=up, in=down] (0,2);
\strand[blue,thick] (5,0) 
to [out=up, in=down] (3,2);
\strand[blue,thick] (3,0) 
to [out=up, in=down] (5,2);
\strand[blue,thick] (6,0) to (6,2);
\strand[blue,thick] (7,0) to (7,2);
\node at (3,-0.5) {(S)};
\strand[blue, thick] (9,0)
    to[out=up, in=down] (9,0.5) 
    to[out=up, in=left] (9.5,1.5)
    to[out=right, in=up]  (10,1)
    to[out=down, in=right]  (9.5,0.5)
    to[out=left, in=down]  (9,1.5)
    to[out=up, in=down]  (9,2);
    \node at (10.5,1){$=t$};
    \strand[blue,thick] (11,0) to (11,2);
    \node at (10.5,-0.5){(T)};
    \strand[blue, thick] (13,0.5)
    to [out=left,in=down] (12.5,1)
    to [out=up,in=left] (13,1.5)
    to [out=right,in=up] (13.5,1)
    to [out=down,in=right] (13,0.5);
    \node at (14.5,1){$=\frac{t-t^{-1}}{z}\unit{\varnothing}$};
    \node at (14,-0.5){(L)};
\end{knot}
\draw[blue,postaction={decorate}] (0,1.95)--(0,2);
\draw[blue,postaction={decorate}] (2,1.95)--(2,2);
\draw[blue,postaction={decorate}] (3,1.95)--(3,2);
\draw[blue,postaction={decorate}] (5,1.95)--(5,2);
\draw[blue,postaction={decorate}] (6,1.95)--(6,2);
\draw[blue,postaction={decorate}] (7,1.95)--(7,2);
\draw[blue,postaction={decorate}] (9,1.95)--(9,2);
\draw[blue,postaction={decorate}] (11,1.95)--(11,2);
\draw[blue,postaction={decorate}] (13.5,1.01)--(13.5,0.99);
\end{tikzpicture}}
\end{center}
 The oriented skein category $\OS$ is defined to be the quotient category of the $\K$-linearization of $\mathcal{FOT}$ with respect to $\mathscr{I}$. With a different normalization of crossings, the category $\OS$ was also introduced
in \cite[\S 5.2]{Tur90} as the Hecke category. Besides, in \cite[Definition 2.1]{QS19} it is called the quantized oriented Brauer category while others may call $\OS$ the framed HOMFLY-PT skein category.

Set
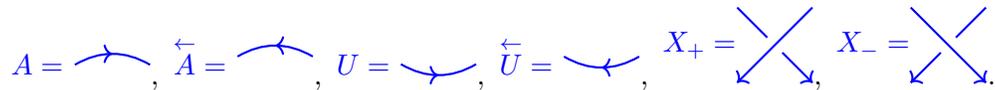
\[
\begin{tikzpicture}[
thick,blue,decoration={
    markings,
    mark=at position 0.5 with {\arrow{>}}}
    ] 
    \node at (-0.5,0) {$A=$};
    \draw[postaction={decorate}] (0,0)to [bend left](1,0);
\end{tikzpicture},\begin{tikzpicture}[
thick,blue,decoration={
    markings,
    mark=at position 0.5 with {\arrow{>}}}
    ] 
    \node at (-0.5,0) {$\opp{A}=$};
    \draw[postaction={decorate}] (1,0)to [bend right](0,0);
\end{tikzpicture},\begin{tikzpicture}[
thick,blue,decoration={
    markings,
    mark=at position 0.5 with {\arrow{>}}}
    ] 
    \node at (-0.5,0) {$U=$};
    \draw[postaction={decorate}] (0,0)to [bend right](1,0);
\end{tikzpicture},
\begin{tikzpicture}[
thick,blue,decoration={
    markings,
    mark=at position 0.5 with {\arrow{>}}}
    ] 
    \node at (-0.5,0) {$\opp{U}=$};
    \draw[postaction={decorate}] (1,0)to [bend left](0,0);
\end{tikzpicture},\begin{tikzpicture}[
thick,blue,decoration={
    markings,
    mark=at position 1 with {\arrow{>}}}
    ] 
    \node at (-0.5,0) {$X_+=$};
    \draw (0,0.5)--(0.4,0.1);
    \draw[postaction={decorate}](0.6,-0.1)--(1,-0.5);
    \draw[postaction={decorate}] (1,0.5)--(0,-0.5);
\end{tikzpicture},
\begin{tikzpicture}[
thick,blue,decoration={
    markings,
    mark=at position 1 with {\arrow{>}}}
    ] 
    \node at (-0.5,0) {$X_-=$};
    \draw (1,0.5)--(0.6,0.1);
    \draw[postaction={decorate}]
        (0.4,-0.1)--(0,-0.5);
    \draw[postaction={decorate}] (0,0.5)--(1,-0.5);
\end{tikzpicture}.
\]
Following a corollary of a more general result of Turaev \cite[Lemma I.3.3]{Tur16}, the monoidal category $\OS$ has a presentation in terms of morphisms $A,\opp{A},U,\opp{U},X_{\pm}$ with relations (cf. \cite{B17})
\begin{align}
\label{eq:rel1}
    (\da A)\circ (U\da)=&\da=(\opp{A}\da)\circ (\da \opp{U}), \\
    (\ua \opp{A})\circ (\opp{U}\ua)=&\ua=(A \ua)\circ(\ua U),\\
    \label{eq:Z}
    (A \ua\ua)\circ(\ua A \da\ua\ua)\circ(\ua\ua X_\pm\ua\ua)\circ&(\ua\ua\da U\ua)\circ(\ua\ua U)= \\
    (\ua\ua\opp{A})\circ (\ua\ua\da\opp{A}\ua)\circ&(\ua\ua X_\pm \ua\ua)\circ(\ua\opp{U}\da\ua\ua)\circ(\opp{U}\ua\ua), \notag\\
    (X_+\da)\circ(\da X_+)\circ(X_+\da)&=(\da X_+)\circ(X_+\da)\circ (\da X_+),\\
    X_\pm \circ X_\mp&=\da\da, \\
    \label{eq:X-} X_+&=X_-+z\da\da,\\
    (\da\opp{A})\circ (X_\pm \ua)\circ(\da U)&=t^{\pm1}\da,\\
    (A\da\ua)\circ(\ua X_\mp\ua)\circ(\ua\da U)&\circ(\ua\da \opp{A})\circ (\ua X_\pm \ua)\circ (\opp{U}\da\ua)=\da\ua,\\
    \label{eq:rel9}
    (\ua\da\opp{A})\circ (\ua X_\pm \ua)\circ (\opp{U}\da\ua)&\circ (A \da\ua)\circ(\ua X_\mp \ua)\circ(\ua\da U)=\ua\da,\\
    \label{eq:rel10}
        A\circ \opp{U}=\opp{A}\circ U&=\frac{t-t^{-1}}{z}.
\end{align}

\subsection{A standard basis}
Given any word $\bf a\in \uda$, we define $\bf a^+$ (resp. $\bf a^-$) to be the number of $\ua$ (resp. $\da$) in the word. For any two words ${\bf a}=a_m\ldots a_1,{\bf b}=b_n\ldots b_1$, an $(\bf a,\bf b)$-matching means a fixed bijection
\[
\{i\mid a_i=\ua\}\cup \{i'\mid b_i=\da\}\overset{\sim}{\to}\{i\mid b_i=\ua\}\cup \{i'\mid a_i=\da\}.
\]
Such a bijection only exists when $\bf a^+-\bf a^-=\bf b^+-\bf b^-$. An $(\bf a,\bf b)$-ribbon can be naturally viewed as a {\em lift} of a given $(\bf a,\bf b)$-matching if the boundary of each strand in the ribbon consists of a pair of points which correspond under the matching; in particular, it contains no “floating bubbles”. An $(\bf a,\bf b)$-ribbon
is called {\em reduced} if no strand crosses itself and no two strands cross more than once.

\begin{proposition}
\cite[Theorem 5.1]{Tur90}
\label{prop:basis}
    The morphism space $Hom_{\OS}(\bf a,\bf b)$ is a free $\K$-module with a basis given by any set consisting of a reduced lift for each of the $(\bf a,\bf b)$-matchings. 
\end{proposition}

In order to construct the canonical basis of the morphism space, we first pin down a specific choice of a set consisting of a reduced lift for each $(\bf a,\bf b)$-matching. In any reduced $(\mathbf{a}, \mathbf{b})$-ribbon, crossings fall into two distinct types. Each crossing is either isotopic to $X^+$ or $X^-$ while maintaining the orientation. Consequently, we refer crossings isotopic to $X^+$ as {\em positive} crossings and those isotopic to $X^-$ as {\em negative} crossings.

\begin{example}
Below is an example of a reduced $(\da\ua,\ua\da\da\ua\da\ua)$-ribbon.
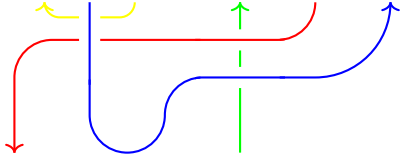
\[
    \begin{tikzpicture}[blue,
thick,decoration={
    markings,
    mark=at position 1 with {\arrow{>}}}
    ] 
    \begin{knot}[clip width=10,
clip radius=15pt,consider self intersections=true,]
\strand[blue,thick] (0,2)
to (0,0.5)
to [out=down,in=left] (0.5,0)
to [out=right,in=down] (1,0.5)
to [out=up,in=left] (1.5,1)
to (3,1)
to [out=right,in=down] (4,2);
\strand[red,thick] (-1,0)
to (-1,1)
to [out=up,in=left] (-0.5,1.5)
to (2.5,1.5)
to [out=right,in=down] (3,2);
\strand[green,thick] (2,2) to (2,0);
\strand[yellow,thick] (-0.6,2)
to [out=down,in=left] (-0.4,1.8)
to (0.4,1.8)
to [out=right,in=down] (0.6,2);
\end{knot}
\draw[postaction={decorate}](4,1.99)--(4,2);
\draw[yellow,postaction={decorate}](-0.6,1.99)--(-0.6,2);
\draw[green,postaction={decorate}](2,1.99)--(2,2);
\draw[red,postaction={decorate}](-1,0.01)--(-1,0);
\end{tikzpicture}
\]
In this reduced ribbon, the  only positive crossing is at the green-blue intersection, while the negative crossings correspond to the blue-yellow, green-red and blue-red intersections.
\end{example}

For each reduced $(\bf a,\bf b)$-matching, there is a unique lift whose crossings are all of positive types, we call such ribbons {\em positive ribbons}. Let $\Di$ denote the set of positive $(\bf a,\bf b)$-ribbons.
\begin{corollary}
    $\Di$ forms a $\K$-basis of the morphism space $Hom_{\OS}(\bf a,\bf b)$.
\end{corollary}

\subsection{A bar involution}
In the case when $\bf a={\bf b}=\ua^m$, the morphism space $Hom_{\OS}(\ua^m,\ua^m)$ is isomorphic to the type A Hecke algebra, provided that the parameters are appropriately specialized. It is well known that the Hecke algebra of type A admits a bar involution and a Kazhdan-Lusztig type basis \cite{KL79}. The following proposition is a natural extension of such a bar involution on the morphism space of the oriented skein category.

\begin{proposition}
\label{prop:bar}
    For any two words ${\bf a}, {\bf b}\in \la \ua,\da\ra$, there is a unique involutive homomorphism $\psi$ on $Hom_{\OS}({\bf a},{\bf b})$ which satisfies
    \begin{equation}
        \psi(t)=t^{-1},\ \psi(z)=-z,\ \psi(A)=A,\ \psi(U)=U,\ \psi(\opp{A})=\opp{A},\ \psi(\opp{U})=\opp{U},\ \psi(X_\pm)=X_\mp.
    \end{equation}
\end{proposition}

\begin{proof}
    It is straightforward to check that $\psi$ preserves the relations \eqref{eq:rel1}--\eqref{eq:rel10}.
\end{proof}

We will refer to $\psi$ as the bar involution on $Hom_{\OS}({\bf a},{\bf b})$. For any $(\mathbf{a},\mathbf{b})$-ribbon $\mathbf{T}\in Hom_{\OS}({\bf a},{\bf b})$, the image $\psi(\mathbf{T})$ corresponds to the reduced lift of the same $(\mathbf{a},\mathbf{b})$-matching, but with all positive and negative crossings swapped.

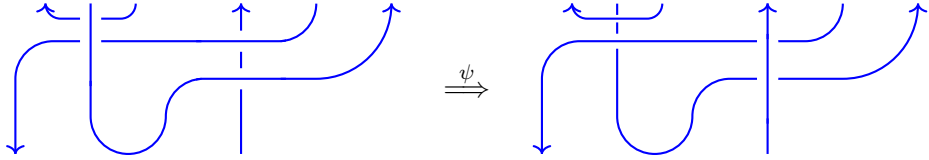
\begin{example}
\[
     \begin{tikzpicture}[blue,
thick,decoration={
    markings,
    mark=at position 1 with {\arrow{>}}}
    ] 
    \begin{knot}[clip width=10,
clip radius=15pt,consider self intersections=true,]
\strand[blue,thick] (0,2)
to (0,0.5)
to [out=down,in=left] (0.5,0)
to [out=right,in=down] (1,0.5)
to [out=up,in=left] (1.5,1)
to (3,1)
to [out=right,in=down] (4,2);
\strand[blue,thick] (-1,0)
to (-1,1)
to [out=up,in=left] (-0.5,1.5)
to (2.5,1.5)
to [out=right,in=down] (3,2);
\strand[blue,thick] (2,2) to (2,0);
\strand[blue,thick] (-0.6,2)
to [out=down,in=left] (-0.4,1.8)
to (0.4,1.8)
to [out=right,in=down] (0.6,2);
\end{knot}
\draw[postaction={decorate}](4,1.99)--(4,2);
\draw[postaction={decorate}](-0.6,1.99)--(-0.6,2);
\draw[postaction={decorate}](2,1.99)--(2,2);
\draw[postaction={decorate}](-1,0.01)--(-1,0);
\node [black] at (5,1){$\overset{\psi}{\Longrightarrow}$};
\begin{knot}[clip width=10,
clip radius=15pt,]
\strand[blue,thick] (6.4,2)
to [out=down,in=left] (6.6,1.8)
to (7.4,1.8)
to [out=right,in=down] (7.6,2);
\strand[blue,thick] (9,2) to (9,0);
\strand[blue,thick] (6,0)
to (6,1)
to [out=up,in=left] (6.5,1.5)
to (9.5,1.5)
to [out=right,in=down] (10,2);
\strand[blue,thick] (7,2) to (7,1.85);
\strand[blue,thick] (7,1.75)
to (7,1.6);
\strand[blue,thick] (7,1.4)
to (7,0.5)
to [out=down,in=left] (7.5,0)
to [out=right,in=down] (8,0.5)
to [out=up,in=left] (8.5,1)
to (10,1)
to [out=right,in=down] (11,2);
\end{knot}
\draw[postaction={decorate}](11,1.99)--(11,2);
\draw[postaction={decorate}](6.4,1.99)--(6.4,2);
\draw[postaction={decorate}](9,1.99)--(9,2);
\draw[postaction={decorate}](6,0.01)--(6,0);
\end{tikzpicture}
\]
\end{example}

It is not hard to see that $\psi$ coincides with the bar involution on the Hecke algebra of type A when $\bf a$ and $\bf b$ are specialized. 

\subsection{Canonical bases}
From now on we suppose $\K=\Q(q,t)$ where $q,t$ are indeterminant and we fix $z=q-q^{-1}$. 

For any ribbon ${\bf T}\in Hom_{\OSqq}({\bf a},{\bf b})$, the length of it is given by:
 \begin{equation}
        \label{eq:length}
        \ell(\bf T):=\#\text {\{crossings in $\bf T$\}}.
    \end{equation}

The following lemma is a crucial step toward the construction of the canonical basis.
\begin{lemma}
\label{lemma:psiT}
    For any $\bf T\in \Di$, we have
$$\psi({\bf T})={\bf T}+\sum\limits_{\substack{{\bf T}'\in  \Di\\\ell({\bf T}')<\ell({\bf T})}}
x_{{\bf T}', {\bf T}}\mathcal {\bf T}'$$
for some $x_{{\bf T}', {\bf T}}\in \Z[q, q^{-1}].$
\end{lemma}

\begin{proof}
    For any $\bf T\in \Di$, $\psi({\bf T})$ is a reduced $(\bf a,\bf b)$-ribbon with only negative crossings. Applying \eqref{eq:X-} to all the crossings of $\psi({\bf T})$, we see that \[ \psi({\bf T})={\bf T}+\sum_{T'\in  Hom_{\OSq}({\bf a},{\bf b})}a_{T',{\bf T}}T' \]for some $a_{T', {\bf T}}\in \K.$ Since the relation \eqref{eq:X-} does not produce new crossings, hence $a_{T',{\bf T}}=0$ unless $\ell(T')<\ell({\bf T})$. Thus we have
    \[ \psi({\bf T})={\bf T}+\sum\limits_{\substack{ T'\in  Hom_{\OSq}(\bf a,\bf b)\\\ell(T')<\ell({\bf T})}}a_{T',{\bf T}}T'. \]
    Since $\Di$ is a basis of $Hom_{\OSq}(\bf a,\bf b)$, we can rewrite $\psi({\bf T})-{\bf T}$ in terms of elements in $\Di$ using relations 
    \eqref{eq:rel1}--\eqref{eq:rel10}. Since ${\bf T}$ is reduced, no bubble will be produced or moved in the process. Hence \eqref{eq:rel10} will not be used. Moreover, all relations \eqref{eq:rel1}--\eqref{eq:rel9} only decrease the number of crossings. Hence we have $$\psi({\bf T})={\bf T}+\sum\limits_{\substack{{\bf T}'\in  \Di\\\ell({\bf T}')<\ell({\bf T})}}
x_{{\bf T}', {\bf T}}\mathcal {\bf T}'$$
for some $x_{{\bf T}', {\bf T}}\in \Z[q, q^{-1}].$
\end{proof}

By Lemma \ref{lemma:psiT} and Lusztig's lemma (cf. \cite[Lemma 24.2.1]{Lu94}), we obtain the canonical basis for $Hom_{\OSqq}(\bf a,\bf b)$ over $\K$.

\begin{theorem}\label{thm:cbhom}
There exists a unique basis of $Hom_{\OSqq}(\bf a,\bf b)$ over $\K$, called the canonical basis, written as $\{\mathcal C_{{\bf T}}|{\bf T}\in  \Di\}$ such that
\begin{align*}
&(1)~~\psi(\mathcal C_{{\bf T}})=\mathcal C_{{\bf T}},\\
&(2)~~\mathcal C_{{\bf T}}=\mathcal {\bf T}+\sum\limits_{\substack{{\bf T}'\in \Di \\ \ell({\bf T}')<\ell({\bf T})}}y_{{\bf T}', {\bf T}}\mathcal {\bf T}', ~\mbox{where } y_{{\bf T}', {\bf T}}\in q^{-1}\Z[q^{-1}].
\end{align*}
\end{theorem}

\begin{remark}
    As in the proof of Lemma~\ref{lemma:psiT}. For any ${\bf T}\in \Di$, the process of expressing $\psi(T)$ in terms of elements in $\Di$ with shorter length does not produce or remove any bubble. Therefore, the transition matrices between the canonical basis and standard basis is independent of the variable $t$. The same phenomenon also appeared in \cite[\S 5.2]{FG95},\cite{CS22}.
\end{remark}

Fix $\bf a,\bf b\in \uda$ and let $\tilde{\bf a}={\bf a}\ua\da$. We define an embedding as follows:
\begin{align*}
i:Hom_{\OSqq}(\bf a,\bf b)&\to Hom_{\OSqq}(\tilde{\bf a},\bf b)\\
{\bf T} &\mapsto {\bf T}\curvearrowright
\end{align*}

\begin{proposition}
\label{cor:embedding}
    For any canonical basis element $\mathcal C_{{\bf T}}\in Hom_{\OSqq}({\bf a,\bf b})$, $i(\mathcal C_{{\bf T}})$ is a canonical basis element of $Hom_{\OSqq}(\tilde{\bf a},\bf b)$.
\end{proposition}

\begin{proof}
    Suppose we have \[\mathcal C_{{\bf T}}=\mathcal {\bf T}+\sum\limits_{\substack{{\bf T}'\in \Di \\ \ell({\bf T}')<\ell({\bf T})}}y_{{\bf T}', {\bf T}}\mathcal {\bf T}'.\] 
   We see that
    \begin{align*}
        i(\mathcal C_{\bf T})=&{\bf T}\curvearrowright+\sum\limits_{\substack{{\bf T}'\in \Di \\ \ell({\bf T}')<\ell({\bf T})}}y_{{\bf T}', {\bf T}}\mathcal {\bf T}'\curvearrowright\\
       &\in {\bf T}\curvearrowright+\sum\limits_{\substack{{\bf T}''\in \mathcal{P}_{\tilde{\bf a},\bf b} \\ \ell({\bf T}'')<\ell({\bf T}\curvearrowright)}}q^{-1}\Z[q^{-1}] {\bf T}''
    \end{align*}
since $\ell({\bf T}'\curvearrowright)=\ell(\bf{T}')<\ell(\bf T)=\ell({\bf T}\curvearrowright)$ and $y_{{\bf T}', {\bf T}}\in q^{-1}\Z[q^{-1}]$. 

On the other hand, since $\bf T$ and $\curvearrowright$ are both $\psi$-invariant, we see that $i(\mathcal C_{\bf T})$ is $\psi$-invariant as well. Thus by the uniqueness of the canonical basis we see that $i(\mathcal C_{{\bf T}})$ is a canonical basis element of $Hom_{\OSqq}(\tilde{\bf a},\bf b)$.
\end{proof}

\begin{remark}
The embedding $i$ can be easily generalized to incorporate more complicated morphism spaces by adding more non-crossing cups and caps. The canonical bases are always preserved under such an embedding, following an inductive proof similar to the proof of Proposition~\ref{cor:embedding}.
\end{remark}

\section{Quantized walled Brauer algebra}
\label{sec:qWB}
In this section, we review the definition and basic properties of the quantized walled Brauer algebras. Moreover, we realize them as specific morphism spaces of $\OSqq$, allowing us to apply the constructions in Proposition~\ref{prop:bar} and Theorem~\ref{thm:cbhom} to the quantized walled Brauer algebras. Although the results in this section do not extend beyond the earlier constructions, we still find it helpful to express the bar involution and the canonical bases within the algebraic framework.
\vspace{-1em}

\subsection{Basic properties}
Fix $q\in \K^*$, $\nl,\nr \in\N$ and set
\begin{equation}
\label{eq:I}
\I_-=\{-\nl, \ldots,-2, -1\},\quad \I_+=\{1,2,\ldots,\nr\},\quad \I=\I_-\cup \I_+.    
\end{equation}
We often denote $-k$ by $\underline{k}$ for $k\geqslant 1$.

\begin{definition} {\rm \cite{Le94}(also see \cite{En03,RS15})}
\label{def:qWB}
    The quantized walled Brauer algebra $\qWB{\nl}{\nr}{q}{t}$ is generated by $H_i,\ i\in \I$ and $e$ over $\K$ subject to the following relations:
    \begin{equation*}
    \begin{aligned}
        &(1)\ (H_i-q)(H_i+q^{-1})=1,\ i\in \I, \\
        &(2)\ H_iH_j=H_jH_i,\ |i-j|>1, \\
        &(3)\ H_iH_{i+1}H_i=H_{i+1}H_iH_{i+1}\ i,i+1\in\I,\\
        &(4)\ H_ie=eH_i,\ |i|\geqslant 2, \\
        &(5)\ eH_{\pm 1}e=te,\\
        &(6)\ e^2=\frac{t-t^{-1}}{q-q^{-1}}e, \\
        &(7) \ eH^{-1}_{\underline1}H_1eH_{\underline1}=eH^{-1}_{\underline1}H_1eH_1,\qquad (8)\ H_{\underline1}eH^{-1}_{\underline1}H_1e=H_1eH^{-1}_{\underline1}H_1e.
    \end{aligned}
\end{equation*}
\end{definition}

It follows from Definition \ref{def:qWB} that the walled $q$-Brauer algebra $\qWB{\nl}{\nr}{q}{t}$ contains type A Hecke algebras $\HA{\nl}$ and $\HA{\nr}$, which are generated by ${H_{\underline{1}},\ldots,H_{\underline{\nl-1}}}$ and ${H_1,\ldots,H_{\nr-1}}$ respectively. We use the notation $\HA{\underline\nl}$ to refer to the subalgebra $\HA{\nl}$.

For $l,r\in \I_-$ or $l,r\in \I_+$ we let
\begin{align*}
H_{l,r}^{+}=\begin{cases}H_lH_{l+1}\cdots H_{r}& \hbox {if } l\leq r \\H_lH_{l-1}\cdots H_{r}& \hbox {if } l>r\end{cases}
,\quad 
H_{l,r}^{-}=\begin{cases}H_l^{-1}H_{l+1}^{-1}\cdots H_{r}^{-1}& \hbox {if } l\leq r,\\H_l^{-1}H_{l-1}^{-1}\cdots H_{r}^{-1}& \hbox {if } l>r\end{cases}.
\end{align*}

We make the convention that $e_{0}=1$. For each $0\leq k\leq \min\{\nl,\nr\}$, we define the elements $e_{k}$ in $\qWB{\nl}{\nr}{q}{t}$ inductively by
$$e_{1}=e\quad \mathrm{and}\quad e_{k+1}=eH_{\underline 1,\underline k}^-H_{1,k}^+e_k .$$

We collect a few identities in $\qWB{\nl}{\nr}{q}{t}$ which will be used in the sequel.
\begin{lemma}
\label{lemma:ek}
    In $\qWB{\nl}{\nr}{q}{t}$ we have
    \begin{equation}
        \begin{aligned}
        &(1)\ e_{k+1}=e_{k}H_{k,1}^+H_{\underline k,\underline 1}^-e,\\
        &(2)\ 
           e_{k}=e_{k}H_iH^{-1}_{\underline i}=H_{\underline i}^{- 1}H_ie_{k},\ 1\leqslant i<k\leqslant \min\{\nl,\nr\}, \\
        &(3)\ e_{k}e_{i}=e_{i}e_{k}=(\frac{t-t^{-1}}{q-q^{-1}})^{i}e_{k},\ 1\leqslant i\leqslant k\leqslant \min\{\nl,\nr\},\\
            &(4)\  H_ie_{k}=e_{k}H_i,\ |i|\geq k+1, \\
            &(5)\ e_{i}H_je_{k}=e_{k}H_je_{i}=t(\frac{t-t^{-1}}{q-q^{-1}})^{i-1}e_{k},\  1\leqslant |j|\leqslant i\leqslant k\leqslant \min\{\nl,\nr\}.
        \end{aligned}
    \end{equation}
\end{lemma}

\begin{proof}
    Relation $(4)$ follows from an easy induction on $k$.

    For $(1)$, when $k=1$ the equality trivially holds. For general $k$ we see that
    \begin{align*}
        e_{k+1}=&eH_{-1}^{-1}\cdots H_{-k}^{-1}H_1\cdots H_k e_{k} \\
        =&(eH_{-1}^{-1}\cdots H_{-k}^{-1}H_1\cdots H_k) (e H_{-1}^{-1}\cdots H_{\underline{k-1}}^{-1}H_1\cdots H_{k-1})\cdots (eH_{-1}^{-1}H_1)e \\
        =&(eH_{-1}^{-1}\cdots H_{\underline{k-1}}^{-1}H_1\cdots H_{k-1}) (e H_{-1}^{-1}\cdots H^{-1}_{-k+2}H_{-k,-k+1}^-H_1\cdots H_{k-2} H_{k,k-1}^+)\cdots \\
        =&e_{k}H_{k,1}^+H_{\underline k,\underline 1}^-e.
    \end{align*}
    
    For $(2)$, from relations $(7)$ and $(8)$ in Definition \ref{def:qWB} we know that $e_{2}=H_{-1}^{-1}H_1e_{2}=e_{2}H_{-1}^{-1}H_1$. Now we use induction on $k$. Suppose the equation holds for $k$. Then by $(1)$ we have
    \begin{align*}
        H_ie_{k+1}=&H_ie_{k}H_{k,1}^+H_{\underline k,\underline 1}^-e.
    \end{align*}
    If $i<k$, then the equality $H_{\underline i}e_{k+1}=H_ie_{k+1}$ holds by the inductive assumption.
    If $i=k$, we have
    \begin{align*}
        H_{k}e_{k+1}=&H_{k}eH_{1,k}^+H_{\underline 1,\underline k}^-e_{k}
        =eH_1\cdots H_{k-2}(H_kH_{k-1}H_{k})H_{\underline 1,\underline k}^-e_{k}\\
        =&eH_1\cdots H_{k-2}(H_{k-1}H_{k}H_{k-1})H_{\underline 1,\underline k}^-e_{k}
        =eH_{1,k}^+H_{\underline 1,\underline k}^- H_{k-1}e_{k}\\
        =&eH_{1,k}^+H_{\underline 1,\underline k}^- H_{\underline{k-1}}e_{k}=eH_{1,k}^+ H_{\underline 1}^{-1}\cdots H_{\underline{k-2}}^{-1}(H_{\underline{k-1}}^{-1}H_{\underline k}^{-1}H_{\underline{k-1}}^{-1})e_{k}\\
        =&eH_{1,k}^+H_{\underline 1,\underline k}^- H_{\underline{k-1}}e_{k}=eH_{1,k}^+ H_{\underline 1}^{-1}\cdots H_{\underline{k-2}}^{-1}(H_{\underline k}^{-1}H_{\underline{k-1}}^{-1}H_{\underline k}^{-1})e_{k}\\
        =&H_{\underline k}eH_{1,k}^+H_{\underline 1,\underline k}^-e_{k}=H_{\underline k}e_{k+1}.
    \end{align*}
    The other equality $e_{k+1}H_i=e_{k+1}H_{-i}$ can be proved similarly.

    For $(3)$, we use induction on $i$. When $i=1$, clearly we have $ee_{k}=e_{k}e=\frac{t-t^{-1}}{q-q^{-1}}e_{k}$. Suppose the equality holds for $j<k$. Then by $(1)$ we have
    \begin{align*}
        e_{j+1}e_{k}=&e_{j}H_{j,1}^+H_{\underline j,\underline 1}^-e e_{k}
        =\frac{t-t^{-1}}{q-q^{-1}}e_{j}H_{j,1}^+H_{\underline j,\underline 1}^-e_{k}\\
        \overset{(2)}{=}&\frac{t-t^{-1}}{q-q^{-1}}e_{j}e_{k}=(\frac{t-t^{-1}}{q-q^{-1}})^{j+1}e_{k}.    \end{align*}

    For $(5)$, when $i=1$, we see that
    \begin{align*}
        eH_{\underline1}e_{k}=eH_{\underline1}eH_{\underline1,\underline{k-1}}^{-}H_{1,k-1}^{+}e_{k-1}=teH_{\underline1,\underline{k-1}}^{-}H_{1,k-1}^{+}e_{k-1}=te_{k}.
    \end{align*}
    and similarly $e_{k}H_{-1}e=te_{k}$.
    Now suppose both qualities in $(5)$ holds for $i$, then for $|j|\leqslant i+1\leqslant k$ we have
    \begin{align*}
        e_{i+1}H_je_{k}=&e H_{\underline 1,\underline i}^-H_{1,i}^+e_{i}H_je_{k}.
    \end{align*}
    If $|j|<i+1$, then the equality holds by the inductive assumption and $(4)$. If $|j|=i+1$, take $j=-i-1$ for example we have
    \begin{align*}
        &e_{i+1}H_{\underline {i+1}}e_{k}\\
        =&e_{i}H_{\underline i,\underline 1}^-H_{i,1}^+ eH_{\underline {i+1}}e_{k}=\frac{t-t^{-1}}{q-q^{-1}}e_{i}H_{\underline i,\underline 1}^-H_{i,1}^+H_{\underline {i+1}}e_{k}\\
        =&\frac{t-t^{-1}}{q-q^{-1}}e_{i}H_{\underline i,\underline 1}^-H_{\underline {i+1}}H_{i,1}^+e_{k}\overset{(2)}{=}\frac{t-t^{-1}}{q-q^{-1}}e_{i}H_{\underline i,\underline 1}^-H_{\underline {i+1}}H_{\underline 1,\underline i}^+e_{k}\\
        =&\frac{t-t^{-1}}{q-q^{-1}}e_{i}H_{\underline i}^{-1}H_{\underline {i+1}}H_{\underline i}e_{k}=\frac{t-t^{-1}}{q-q^{-1}}e_{i}(H_{\underline i}+(q^{-1}-q))H_{\underline {i+1}}H_{\underline i}e_{k}.
    \end{align*}
    Moreover, by the inductive assumption, we see that\[e_{i}H_{\underline {i+1}}H_{\underline i}e_{k}=H_{\underline {i+1}}e_{i}H_{\underline i}e_{k}=t(\frac{t-t^{-1}}{q-q^{-1}})^{i-1}H_{\underline {i+1}}e_{k},\]
    while
    \begin{align*}
&e_{i}H_{\underline i}H_{\underline {i+1}}H_{\underline i}e_{k}=e_{i}H_{\underline {i+1}}H_{\underline i}H_{\underline {i+1}}eH_{\underline 1,\underline {k-1}}^-H_{1,k-1}^+e_{k-1}\\
        =&H_{\underline {i+1}}(e_{i}H_{\underline i}e)H_{\underline {i+1}}H_{\underline 1,\underline {k-1}}^-H_{1,k-1}^+e_{k-1}\\
        =&\frac{q-q^{-1}}{t-t^{-1}}H_{\underline {i+1}}(e_{i}H_{\underline i}e)e H_{\underline {i+1}}H_{\underline 1,\underline {k-1}}^-H_{1,k-1}^+e_{k-1}\\ 
        =&t\frac{q-q^{-1}}{t-t^{-1}}H_{\underline {i+1}}^2e_{i}e_{k}
        =t(\frac{t-t^{-1}}{q-q^{-1}})^{i-1}(1+(q-q^{-1})H_{\underline {i+1}})e_{k}.
    \end{align*}
    Thus we conclude that $e_{i+1}H_{\underline {i+1}}e_{k}=t(\frac{t-t^{-1}}{q-q^{-1}})^{i}e_{k}$.
    The case $j=i+1$ can be proved similarly.
\end{proof}

\begin{remark}
    Applying the relations in Lemma \ref{lemma:ek}, we observe that the idempotents $e_{k}$ here are equal to the idempotents $e^k$ defined in \cite[Lemma 2.8]{RS15}.
\end{remark}

\subsection{A bar involution}
 We label the top row and the bottom row of any $(\ua^\nl\da^\nr,\ua^\nl\da^\nr)$-ribbon by $\I$ in increasing order from left to right. 

\begin{proposition}\cite{Tur90}
\label{prop:qWBhom}
    There is an algebra isomorphism \[\Xi:\qWB{\nl}{\nr}{q}{t}\to End_{\OSqq}(\ua^m\da^n)\] given by (unessential labels are omitted)
\begin{equation*}
\begin{aligned}
      &\begin{tikzpicture}[blue,
thick,decoration={
    markings,
    mark=at position 1 with {\arrow{>}}}
    ] 
    \node at (-0.5,0) {$\Xi(H_{\underline{i}})=$};
    \draw[postaction={decorate}]  (0.65,-0.5)--(0.65,0.5);
    \node at (1,0) {$\cdots$};
    \draw (2.5,-0.5)--(2.1,-0.1);
    \draw[postaction={decorate}](1.9,0.1)--(1.5,0.5);
    \draw[postaction={decorate}] (1.5,-0.5)--(2.5,0.5);
    \node at (2.75,0) {$\cdots$};
    \draw[postaction={decorate}]  (3.25,-0.5)--(3.25,0.5);
     \draw[postaction={decorate}]  (3.55,0.5)--(3.55,-0.5);
     \node at (2.5,-0.75) {$\underline{i}$};
     \node at (1.5,-0.75) {$\underline{i+1}$};
     \node at (2.5,0.75) {$\underline{i}$};
     \node at (1.5,0.75) {$\underline{i+1}$};
     \node at (3.9,0) {$\cdots$};
     \draw[postaction={decorate}]  (4.25,0.5)--(4.25,-0.5);
  \end{tikzpicture},\quad
      \begin{tikzpicture}[blue,
thick,decoration={
    markings,
    mark=at position 1 with {\arrow{>}}}
    ] 
    \node at (-0.5,0) {$\Xi(H_{j})=$};
    \draw[postaction={decorate}]  (0.65,-0.5)--(0.65,0.5);
    \node at (1,0) {$\cdots$};
    \draw[postaction={decorate}]  (1.35,-0.5)--(1.35,0.5);
      \draw[postaction={decorate}]  (1.55,0.5)--(1.55,-0.5);
      \node at (2,0) {$\cdots$};
    \draw (2.2,0.5)--(2.6,0.1);
    \draw[postaction={decorate}](2.8,-0.1)--(3.2,-0.5);
    \draw[postaction={decorate}] (3.2,0.5)--(2.2,-0.5);
    \node at (2.2,-0.75) {$j$};
    \node at (3.2,-0.75) {$j+1$};
      \node at (2.2,0.75) {$j$};
    \node at (3.2,0.75) {$j+1$};
     \node at (3.6,0) {$\cdots$};
     \draw[postaction={decorate}]  (3.95,0.5)--(3.95,-0.5);
  \end{tikzpicture},\\
      &\begin{tikzpicture}[blue,
thick,decoration={
    markings,
    mark=at position 1 with {\arrow{>}}}
    ] 
    \node at (-0.5,0) {$\Xi(e)=$};
    \draw[postaction={decorate}]  (0.65,-0.5)--(0.65,0.5);
    \node at (1,0) {$\cdots$};
    \draw[postaction={decorate}]  (1.35,-0.5)--(1.35,0.5);
    \draw[postaction={decorate}](2.8,0.5) to [bend left](1.8,0.5);
    \draw[postaction={decorate}] (1.8,-0.5) to [bend left](2.8,-0.5);
     \draw[postaction={decorate}]  (3.25,0.5)--(3.25,-0.5);
      \node at (2.8,0.75) {$1$};
      \node at (1.8,0.75) {$\underline{1}$};
       \node at (2.8,-0.75) {$1$};
      \node at (1.8,-0.75) {$\underline{1}$};
     \node at (3.6,0) {$\cdots$};
     \draw[postaction={decorate}]  (3.95,0.5)--(3.95,-0.5);
  \end{tikzpicture}.
\end{aligned}
\end{equation*}
\end{proposition}

It follows from Theorem~\ref{thm:cbhom} and Proposition~\ref{prop:qWBhom} that the quantized walled Brauer algebra admits both a bar-involution and a Kazhdan-Lusztig type basis. The following lemma decodes the information of the bar involution $\psi$ in Proposition~\ref{prop:bar} on the algebra level.

\begin{lemma}\label{lemma:bar}
There is a unique involutive homomorphism $\overline{\ \cdot\ }$ on $\qWB{\nl}{\nr}{q}{t}$ which is $\Q$-linear and satisfies $\overline q=q^{-1}, \overline {t}=t^{-1}, \overline{H_{i}}=H_{i}^{-1}$ and $\overline{e}=e$.
\end{lemma}

\begin{proof}
    It suffices to check that  $\overline{\ \cdot\ }$ preserves the defining relations in Definition \ref{def:qWB}. It is straightforward to check that $\overline{\ \cdot\ }$ preserves the relations (1)--(6). 
    
    To check $\overline{\ \cdot\ }$ preserves the relation (7), we first show that $e_{2}=eH_1H_{\underline1}^{-1}e$ is bar invariant. On one hand, we have $e_{2}=eH_1e+(q^{-1}-q)eH_1H_{\underline1}e=ze+(q^{-1}-q)eH_1H_{\underline1}e$. On the other hand, 
    \[\overline{e_{2}}=eH_{\underline1}H_1^{-1}e=eH_{\underline1}e+(q^{-1}-q)eH_1H_{\underline1}e=ze+(q^{-1}-q)eH_1H_{\underline1}e.\]
    Thus we see that $\overline{e_{2}}=e_{2}$.
    Now relation (7) in Definition \ref{def:qWB} is equivalent to 
    \[e_{2}H_1H_{\underline1}^{-1}=e_{2}.\]
    Apply $\overline{\ \cdot\ }$ on the left hand side and apply relations in Lemma \ref{lemma:ek} we see that
    \[\overline{e_{2}H_1H_{\underline1}^{-1}}=e_{2}H^{-1}_1H_{\underline1}=e_{2}=\overline{e_{2}}.\]
    Thus the relation (7) is preserved by $\overline{\ \cdot\ }$. Similar proof works for checking  the relation (8).
\end{proof}

In the proof of Lemma \ref{lemma:bar}, we already see that $e_{2}$ is bar invariant. In fact, $e_k$ corresponds to the following ribbon:
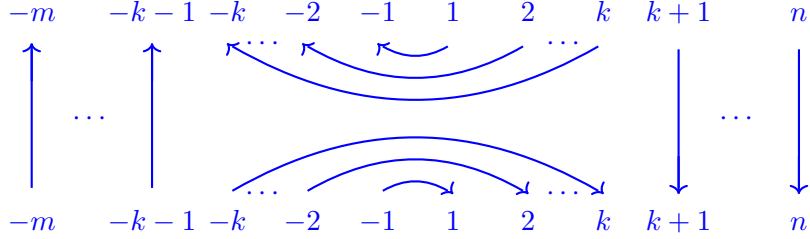
\begin{equation*}
    \begin{tikzpicture}[blue,scale=1,thick,decoration={
    markings,
    mark=at position 1 with {\arrow{>}}}]
\node (1) at (0.2,1){$\cdots$};
\node (x1) [circle,scale=0.4,label={$-k-1$}] at (1,2){};
\node (x2) [circle,scale=0.4,label=below:{$-k-1$}] at (1,0){};
\node (x3) [circle,scale=0.4,label={$-\nl$}] at (-0.6,2){};
\node (x4) [circle,scale=0.4,label=below:{$-\nl$}] at (-0.6,0){};
\node (2) [circle,scale=0.4,label={$-k$}] at (2,2){};
\node (2.5) at (2.5,2){$\cdots$};
\node (3.5) at (2.5,0){$\cdots$};
\node (3) [circle,scale=0.4,label={$-2$}] at (3,2){};
\node (4) [circle,scale=0.4,label={$-1$}] at (4,2){};
\node (5) [circle,scale=0.4,label={$1$}] at (5,2){};
\node (6) [circle,scale=0.4,label={$2$}] at (6,2){};
\node (7) [circle,scale=0.4,label={${k}$}] at (7,2){};
\node (10) [circle,scale=0.4,label=below:{$-k$}] at (2,0){};
\node (11) [circle,scale=0.4,label=below:{$-2$}] at (3,0){};
\node (12) [circle,scale=0.4,label=below:{$-1$}] at (4,0){};
\node (13) [circle,scale=0.4,label=below:{$1$}] at (5,0){};
\node (14) [circle,scale=0.4,label=below:{$2$}] at (6,0){};
\node (15) [circle,scale=0.4,label=below:{${k}$}] at (7,0){};
\node (17) at (6.5,2){$\cdots$};
\node (17.5) at (6.5,0){$\cdots$};
\node (y1)[circle,scale=0.4,label={$k+1$}] at (8,2){};
\node (y2)[circle,scale=0.4,label=below:{$k+1$}] at (8,0){};
\node (y3)[circle,scale=0.4,label={$\nr$}] at (9.6,2){};
\node (y4)[circle,scale=0.4,label=below:{$\nr$}] at (9.6,0){};
\node (8) at (8.8,1){$\cdots$};


\path 
          (2) edge[bend right] (7)
          (3) edge[bend right] (6)
          (4) edge[bend right] (5)
          (10) edge[bend left] (15)
          (11) edge[bend left] (14)
          (12) edge[bend left] (13)
          (x1) edge (x2)
          (y1) edge (y2)
          (x3) edge (x4)
          (y3) edge (y4);
    \draw[postaction={decorate}] (-0.6,1.5)--(-0.6,2);
    \draw[postaction={decorate}] (1,1.5)--(1,2);
    \draw[postaction={decorate}] (8,0.5)--(8,0);
    \draw[postaction={decorate}] (9.6,0.5)--(9.6,0);
    \draw[postaction={decorate}] (2.01,1.99)--(2,2);
     \draw[postaction={decorate}] (3.01,1.99)--(3,2);
      \draw[postaction={decorate}] (4.01,1.99)--(4,2);
     \draw[postaction={decorate}] (6.99,0.01)--(7,0);
      \draw[postaction={decorate}] (5.99,0.01)--(6,0);
       \draw[postaction={decorate}] (4.99,0.01)--(5,0);
\end{tikzpicture}
\end{equation*}
which is $\psi$-invariant following from the Proposition~\ref{prop:bar}. Here we also give an algebraic proof of this observation.
\begin{lemma}
     We have $\overline{e_{k}}=e_{k}$ for $0\leq k\leq \min\{\nl,\nr\}$.
\end{lemma}

\begin{proof}
     We have
     \begin{equation*}
         \begin{aligned}
             e_{k+1}=&eH_{\underline 1,\underline k}^{-}H_{1,k}^{+}e_{k}\\
             =&(\frac{q-q^{-1}}{t-t^{-1}})^{k-1}eH_{\underline 1,\underline k}^{-}H_{1,k}^{+}\cdot e_{k-1}e_{k}\\
             =&(\frac{q-q^{-1}}{t-t^{-1}})^{k-1}eH_{\underline 1,\underline {k-1}}^{-}H_{1,k-1}^{+}e_{k-1}\cdot H_{ \underline k}^{-1}H_ke_{k}\\
             =&(\frac{q-q^{-1}}{t-t^{-1}})^{k-1}e_{k}H_{\underline k}^{-1}H_ke_{k}.
         \end{aligned}
     \end{equation*}
    
    Then we use induction on $k$ to prove the lemma. The case $k=0,1$ follows from the definition of $\overline{\ \cdot\ }$.
    The case $k=2$ follows from the proof of Lemma \ref{lemma:bar}.
    Now suppose $\overline{e_{k}}=e_{k}$ for $k\geqslant 3$,
    we see that
    \begin{align*}
        \overline{e_{k+1}}=&\overline{(\frac{q-q^{-1}}{t-t^{-1}})^{k-1}e_{k}H_{\underline k}^{-1}H_ke_{k}}=(\frac{q-q^{-1}}{t-t^{-1}})^{k-1}e_{k}H_{\underline k}H^{-1}_ke_{k}.
    \end{align*}
    Thus to prove $\overline{e_{k+1}}=e_{k+1}$ it suffices to prove that $e_{k}H_{\underline k}H^{-1}_ke_{k}=e_{k}H_{\underline k}^{-1}H_ke_{k}$.
   We see that 
    \[e_{k}H_{\underline k}H^{-1}_ke_{k}=e_{k}H_{\underline k}e_{k}+(q^{-1}-q)e_{k}H_{\underline k}H_ke_{k}\]
    while
    \[e_{k}H^{-1}_{\underline k}H_ke_{k}=e_{k}H_{k}e_{k}+(q^{-1}-q)e_{k}H_{\underline k}H_ke_{k}.\]
    Since $e_{k}H_{k}e_{k}=e_{k}H_{\underline k}e_{k}$ by Lemma \ref{lemma:ek}, we conclude that $\overline{e_{k+1}}=e_{k+1}$.
\end{proof}

\subsection{Canonical bases}
Following from Theorem~\ref{thm:cbhom} and Proposition~\ref{prop:qWBhom}, we see that the Kazhdan-Lusztig type basis (canonical basis) of the quantized walled Brauer algebra $\qWB{\nl}{\nr}{q}{t}$ is parameterized by $(\ua^\nl\da^\nr,\ua^\nl\da^\nr)$-matchings. Here we give examples of the canonical bases in low-rank cases.

\begin{example}
\label{ex:n2m1}
    Suppose $\nl=2$ and $\nr=1$, then the canonical basis of $\qWB{2}{1}{q}{t}$ is given by 
    \begin{align*}
        &C_{0}=id,\ C_{s_{\underline1}}=H_{\underline1}+q^{-1}, \\
        &C_{e}=e,\  C_{s_{\underline1}e}=H_{\underline1}e+q^{-1}e,\ C_{es_{\underline1}}=eH_{\underline1}+q^{-1}e,\\
    &C_{s_{\underline1}es_{\underline1}}=H_{\underline1}eH_{\underline1}+q^{-1}eH_{\underline1}+q^{-1}H_{\underline1}e+q^{-2}e.
    \end{align*}
\end{example}

\begin{example}
    Suppose $\nl=2$ and $\nr=2$, then the canonical basis of $\qWB{2}{2}{q}{t}$ is given by 
    \begin{align*}
        &C_{0}=id,\ C_{s_{\underline1}}=H_{\underline1}+q^{-1},\ C_{s_{1}}=H_{1}+q^{-1},\\
        &C_{s_1s_{\underline1}}=H_{\underline1}H_{1}+q^{-1}H_{\underline1}+q^{-1}H_1+q^{-2},\\
        &C_{e}=e,\  C_{s_{\underline1}e}=H_{\underline1}e+q^{-1}e,\ C_{s_{1}e}=H_{1}e+q^{-1}e, \\&C_{s_1s_{\underline1}e}=H_{\underline1}H_{1}e+q^{-1}H_{\underline1}e+q^{-1}H_1e+q^{-2}e,\\
        &C_{es_{\underline1}}=eH_{\underline1}+q^{-1}e,\ C_{es_{1}}=eH_{1}+q^{-1}e,\\
        &C_{es_1s_{\underline1}}=eH_{\underline1}H_{1}+q^{-1}eH_{\underline1}+q^{-1}eH_1+q^{-2}e,\\
    &C_{s_{\underline 1}es_{\underline 1}}=H_{\underline1}eH_{\underline1}+q^{-1}eH_{\underline1}+q^{-1}H_{\underline1}e+q^{-2}e,\\  
    &C_{s_{\underline1}es_{1}}=H_{\underline1}eH_{1}+q^{-1}eH_{1}+q^{-1}H_{\underline1}e+q^{-2}e,\\
    &C_{s_{1}es_{\underline1}}=H_{1}eH_{\underline1}+q^{-1}eH_{\underline1}+q^{-1}H_{1}e+q^{-2}e,\\
    &C_{s_{1}es_{1}}=H_{1}eH_{1}+q^{-1}eH_{1}+q^{-1}H_{1}e+q^{-2}e,\\
&C_{s_{\underline1}es_{\underline1}s_{1}}=C_{s_{\underline1}es_{\underline1}}C_{s_{1}}=C_{s_{\underline1}}C_{es_{\underline1}s_1},\\
&C_{s_{1}es_{1}s_{\underline1}}=C_{s_{1}es_{1}}C_{s_{\underline1}}=C_{s_{1}}C_{es_{\underline1}s_1},\\
&C_{s_{\underline1}s_1es_{1}}=C_{s_{\underline1}}C_{s_{1}es_{1}}=C_{s_{\underline1}s_1e}C_{s_{1}},\\
&C_{s_{\underline1}s_1es_{\underline1}}=C_{s_{\underline1}}C_{s_{1}es_{\underline1}}=C_{s_{\underline1}s_1e}C_{s_{\underline1}},\\
&C_{s_{\underline1}s_1es_{1}s_{\underline1}}=C_{s_{1}s_{\underline1}}C_eC_{s_{1}s_{\underline1}},\\
&C_{e_{2}}=e_{2}, \ C_{s_{\underline1}e_{2}}=C_{s_{\underline1}}e_{2},\ C_{e_{2}s_{\underline1}}=e_{2}C_{s_{\underline1}},\\
& C_{s_{\underline1}e_{2}e_{\underline1}}=C_{s_{\underline1}}e_{2}C_{s_{\underline1}}.
    \end{align*}
\end{example}

\begin{example}
\label{ex:n3m1}
    Suppose $\nl=3$ and $\nr=1$, then the canonical basis of $\qWB{3}{1}{q}{t}$ is given by 
    \begin{align*}
        &C_{0}=id,\ C_{s_{\underline1}}=H_{\underline1}+q^{-1},\ C_{s_{\underline2}}=H_{\underline2}+q^{-1}, \\
        &C_{s_{\underline1}s_{\underline 2}}=C_{s_{\underline1}}C_{s_{\underline2}},\ C_{s_{\underline2}s_{\underline 1}}=C_{s_{\underline2}}C_{s_{\underline1}},\\
        &C_{s_{\underline1}s_{\underline2}s_{\underline 1}}=H_{\underline1}H_{\underline2}H_{\underline1}+q^{-1}H_{\underline1}H_{\underline2}+q^{-1}H_{\underline2}H_{\underline1}+q^{-2}H_{\underline1}+q^{-2}H_{\underline2}+q^{-3},\\
        &C_{e}=e,\  C_{s_{\underline1}e}=H_{\underline1}e+q^{-1}e,\ C_{s_{\underline2}e}=H_{\underline2}e+q^{-1}e,\\
        & C_{s_{\underline1}s_{\underline 2}e}=C_{s_{\underline1}s_{\underline 2}}C_e,\ C_{s_{\underline2}s_{\underline 1}e}=C_{s_{\underline2}s_{\underline 1}}C_e,\ C_{s_{\underline 1}s_{\underline2}s_{\underline 1}e}=C_{s_{\underline 1}s_{\underline2}s_{\underline 1}}C_e,\\
        &C_{es_{\underline1}}=eH_{\underline1}+q^{-1}e,\ C_{es_{\underline1}s_{\underline2}}=C_eC_{s_{\underline1}s_{\underline2}},\ C_{es_{\underline 1}s_{\underline2}s_{\underline 1}}=C_eC_{s_{\underline 1}s_{\underline2}s_{\underline 1}},\ C_{es_{\underline2}s_{\underline1}}=C_eC_{s_{\underline2}s_{\underline1}},\\
    &C_{s_{\underline1}es_{\underline1}}=H_{\underline1}eH_{\underline1}+q^{-1}eH_{\underline1}+q^{-1}H_{\underline1}e+q^{-2}e=C_{s_{\underline 1}}C_eC_{s_{\underline 1}},\\
    &C_{s_{\underline1}es_{\underline1}s_{\underline2}}=C_{s_{\underline1}es_{\underline1}}C_{s_{\underline2}},\\
    &C_{s_{\underline1}es_{\underline1}s_{\underline2}s_{\underline1}}=C_{s_{\underline1}}C_eC_{s_{\underline1}s_{\underline2}s_{\underline1}},\\
    &C_{s_{\underline2}s_{\underline1}es_{\underline1}}=C_{s_{\underline2}s_{\underline 1}}C_eC_{s_{\underline 1}},\\
    &C_{s_{\underline2}s_{\underline1}es_{\underline1}s_{\underline2}}=C_{s_{\underline2}s_{\underline1}}C_eC_{s_{\underline1}s_{\underline2}},\\
    &C_{s_{\underline2}s_{\underline1}es_{\underline1}s_{\underline2}s_{\underline1}}=C_{s_{\underline2}s_{\underline1}}C_eC_{s_{\underline1}s_{\underline2}s_{\underline1}},\\
    &C_{s_{\underline1}s_{\underline2}s_{\underline1}es_{\underline1}}=C_{s_{\underline1}s_{\underline2}s_{\underline 1}}C_eC_{s_{\underline 1}},\\
    &C_{s_{\underline1}s_{\underline2}s_{\underline1}es_{\underline1}s_{\underline2}}=C_{s_{\underline1}s_{\underline2}s_{\underline1}}C_eC_{s_{\underline1}s_{\underline2}}.\\
    \end{align*}
\end{example}

\begin{remark}
    We observe that the canonical bases in the above examples have positivity.
\end{remark}


\begin{thebibliography}{ABCDE90}


\bibitem[B17]{B17}
J.~Brundan,
{\em Representations of the oriented skein category}, \href{https://arxiv.org/abs/1712.08953}{ 	arXiv:1712.08953}





\bibitem[BCHLLS]{BCHLLS}
G.~Benkart, M.~Chakrabarti, T.~Halverson, R.~Leduc, C.~Lee and J.~Stroomer,
{\em Tensor product representations of general linear groups and their connections with
Brauer algebras}, J. Algebra {\bf 166} (1994), 529--567.









\bibitem[CS22]{CS22} W.~Cui and Y.~Shen,
{\em Canonical basis of $q$-Brauer algebras and $\imath$Schur dualities}, Math. Res. Lett. (to appear),
\href{https://arxiv.org/pdf/2203.02082.pdf}{arXiv:2203.02082}



\bibitem[DDS14]{DDS14}
R.~Dipper, S.~Doty and F.~Stoll,
{\em The quantized walled Brauer algebra and mixed tensor space}, Algebr. Represent. Theory {\bf 17} (2014), 675--701.




\bibitem[En03]{En03}
J.~Enyang, 
{\em Cellular bases of the two-parameter version of the centralizer algebra for the mixed
tensor representations of the quantum general linear group},  Combinatorial representation theory
and related topics (Japanese) (Kyoto, 2002), Surikaisekikenkyusho Kokyuroku {\bf 1310} (2003), 134--
153.


\bibitem[FG95]{FG95} S.~Fishel and I.~Grojnowski,
{\em Canonical bases for the Brauer centralizer algebra}, Math. Res. Lett. {\bf 2} (1995), 15--26.


\bibitem[Ko89]{Ko89}
K.~Koike,
{\em On the decomposition of tensor products of the representations of classical groups: by means of universal characters}, Adv. Math. {\bf 74} (1989), 57--86.

\bibitem[KL79]{KL79} D. Kazhdan and G. Lusztig,
{\em Representations of Coxeter groups and Hecke algebras},
Invent. Math. {\bf 53} (1979), 165--184.

\bibitem[KM93]{KM93}
M.~Kosuda and J.~Murakami,
{\em Centralizer algebras of the mixed tensor representations of quantum group $\U_q(\gl(n,\C))$},
Osaka J. Math. {\bf 30} (1993), 475--507. 

\bibitem[Le94]{Le94}
R.~Leduc, 
{\em A two-parameter version of the centralizer algebra of the mixed tensor representation
of the general linear group and quantum general linear group}, thesis, University of Wisconsin-
Madison, (1994).

\bibitem[Lu94]{Lu94} G.~Lusztig,
 {\em Introduction to quantum groups},
Modern Birkh\"auser Classics, Reprint of the 1993 Edition,
Birkh\"auser, Boston, 2010.



\bibitem[QS19]{QS19}
H.~Queffelec and A.~Sartori,
{\em Mixed quantum skew Howe duality and link invariants of type A}, J. Pure Appl. Algebra {\bf 223} (2019), 2733–-2779. 


\bibitem[RS15]{RS15}
H.~Rui and L.~Song,
{\em The representations of quantized walled Brauer algebras}, J. Pure Appl. Algebra {\bf 219} (2015), 1496--1518. 



\bibitem[Tur90]{Tur90}
V.~Turaev,
{\em Operator invariants of tangles, and $R$-matrices}, Math. USSR Izvestiya {\bf 35} (1990), 411--444.

\bibitem[Tur16]{Tur16}
V.~Turaev,
{\em Quantum Invariants of Knots and 3-Manifolds}, De Gruyter Studies in Mathematics {\bf 18}, De Gruyter, Berlin, (2016).

\bibitem[We12]{We12} H.~Wenzl,
{\em A $q$-Brauer algebra}, J. Algebra {\bf 358} (2012), 102--127.

\end{thebibliography}
\end{document}